\documentclass[english]{amsart}
\usepackage[T1]{fontenc}
\usepackage[latin9]{inputenc}
\usepackage{fancyhdr}
\usepackage{amsthm}
\usepackage{amssymb}

\makeatletter
\numberwithin{equation}{section}
\numberwithin{figure}{section}
\newenvironment{lyxlist}[1]
{\begin{list}{}
{\settowidth{\labelwidth}{#1}
 \setlength{\leftmargin}{\labelwidth}
 \addtolength{\leftmargin}{\labelsep}
 }}
{\end{list}}
  \theoremstyle{definition}
  \newtheorem{defn}{\protect\definitionname}
  \theoremstyle{plain}
  \newtheorem{lem}{\protect\lemmaname}
  \theoremstyle{plain}
  \newtheorem{prop}{\protect\propositionname}
  \theoremstyle{plain}
  \newtheorem{cor}{\protect\corollaryname}

 \usepackage{fancyhdr}

\usepackage{babel}

  \providecommand{\definitionname}{Definition}
  \providecommand{\propositionname}{Proposition}
  
\providecommand{\corollaryname}{Corollary}

\usepackage{upgreek}
\usepackage{amssymb}
\let\originalleft\left
\let\originalright\right
\renewcommand{\left}{\mathopen{}\mathclose\bgroup\originalleft}
\renewcommand{\right}{\aftergroup\egroup\originalright}

\makeatother

\usepackage{babel}
  \providecommand{\definitionname}{Definition}
  \providecommand{\lemmaname}{Lemma}
  \providecommand{\propositionname}{Proposition}
\providecommand{\corollaryname}{Corollary}

\begin{document}

\title{On scalable monoids}

\author{Dan Jonsson}

\address{University of Gothenburg, SE-405 30 Gothenburg, Sweden}

\email{dan.jonsson@gu.se}
\begin{abstract}
This brief exposition presents some basic properties of scalable monoids
and quantity spaces, introduced in \cite{key-4}. 
\end{abstract}

\maketitle

\section{Introduction}

Let $X$ be a non-empty set and $R$ a commutative unital ring. An
associative $R$-algebra with $X$ as carrier set combines three operations
on $X$: 
\begin{lyxlist}{00.00.0000}
\item [{(1)}] \emph{addition} of elements of $X$, a binary operation $+:X\times X\rightarrow X$,
written $\left(x,y\right)\mapsto x+y$, such that $\left(X,+\right)$
is an abelian group; 
\item [{(2)}] \emph{multiplication} of elements of $X$, a binary operation
$\cdot:X\times X\rightarrow X$, written $\left(x,y\right)\mapsto x\cdot y$
or $\left(x,y\right)\mapsto xy$, such that $\left(X,\cdot\right)$
is a monoid; 
\item [{(3)}] \emph{scalar multiplication} of elements of $X$ by elements
of $R,$ a monoid action $R\times X\rightarrow X$, written $\left(\alpha,x\right)\mapsto\alpha x$,
where the multiplicative monoid of $R$ acts on $X$, so that $1x=x$
and $\alpha\left(\beta x\right)=\left(\alpha\beta\right)x$. 
\end{lyxlist}
These structures are linked pairwise: 
\begin{lyxlist}{00.00.0000}
\item [{(a)}] addition and multiplication are linked by the distributive
laws $x\left(y+z\right)=xy+xz$ and $\left(x+y\right)z=xz+yz$; one
of these laws suffices when multiplication in $X$ is commutative; 
\item [{(b)}] addition and scalar multiplication are linked by the distributive
laws $\alpha\left(x+y\right)=\alpha x+\alpha y$ and $\left(\alpha+\beta\right)x=\alpha x+\beta x$; 
\item [{(c)}] multiplication and scalar multiplication are linked by the
laws $\alpha\left(xy\right)=\left(\alpha x\right)y$ and $\alpha\left(xy\right)=x\left(\alpha y\right)$;
one of these laws suffices when multiplication in $X$ is commutative. 
\end{lyxlist}
Related algebraic systems can be obtained from associative $R$-algebras
by removing one of the three operations and hence the links between
the removed operation and the two others. Two cases are very familiar.
A unital\emph{ ring} has only addition and multiplication of elements,
linked as described in (a). A\emph{ module} has only addition and
scalar multiplication of elements, linked as described in (b). The
question arises whether it makes sense to define an ``algebra without
an additive group'', with only multiplication and scalar multiplication,
linked as described in (c). Would such a notion be of mathematical
interest and useful for applications? Some reasons to give an affirmative
answer to that question are found in this article.

\section{Basic notions}

It is tempting to call an ``algebra without an additive group''
an ``alebra'', but that may sound somewhat frivolous, and considering
the nature of this algebraic structure, ``scalable monoid'' \cite{key-4},
abbreviated ``scaloid'', might be a better name. This structure
can be defined more or less broadly. In the definition below, multiplication
in $R$ is assumed to be commutative, but multiplication in $X$ is
not. 
\begin{defn}
\label{thm:def1}A \emph{scalable monoid}, or \emph{scaloid,} is an
algebraic structure incorporating structures described by (2), (3)
and (c) above.

More explicitly, a scalable monoid $X$ over a commutative unital
ring $R$ is a monoid $X$ such that there is a function 
\[
\sigma:R\times X\rightarrow X,\qquad\left(\alpha,x\right)\mapsto\alpha x,
\]
called \emph{scalar multiplication}, such that for any $\alpha,\beta\in R$
and any $x,y\in X$ we have 
\begin{enumerate}
\item $1x=x$, 
\item $\alpha\left(\beta x\right)=\left(\alpha\beta\right)x$, 
\item $\alpha\left(xy\right)=\left(\alpha x\right)y=x\left(\alpha y\right)$. 
\end{enumerate}
For clarity, we may denote the unit element of the monoid $X$ by
$1\!_{X}$ or $\mathbf{1}$ and the unit element of the ring $R$
by $1\!_{R}$. $\hphantom{\square}\square$ 
\end{defn}
An \emph{invertible} element of a scaloid $X$ is an element $x\in X$
which has a (necessarily unique) \emph{inverse} $x^{-1}\in X$ such
that $xx^{-1}=x^{-1}x=1\!_{X}$. We can define positive powers of
$x$, denoted $x^{k}$, in the usual way; if $x$ is invertible, negative
powers of $x$ can be defined by setting $x^{-k}=\left(x^{-1}\right)^{k}$.
By convention, $x^{0}=1\!_{X}$. By associativity, $x^{k}x^{\ell}=x^{\left(k+\ell\right)}$.
A product of invertible elements is clearly an invertible element.

The following two simple facts about scaloids will be used repeatedly: 
\begin{lem}
\label{thm:lem1}Let $X$ be a scalable monoid over $R$. 
\begin{enumerate}
\item $\left(\alpha x\right)\left(\beta y\right)=\left(\alpha\beta\right)\left(xy\right)$
for any $x,y\in X$ and $\alpha,\beta\in R$. 
\item $\alpha\left(\beta x\right)=\beta\left(\alpha x\right)$ and $\left(\alpha\beta\right)x=\left(\beta\alpha\right)x$
for any $x\in X$ and $\alpha,\beta\in R$. 
\end{enumerate}
\end{lem}
\begin{proof}
(1). $\left(\alpha x\right)\left(\beta y\right)=\alpha\left(x\left(\beta y\right)\right)=\alpha\left(\beta\left(xy\right)\right)=\left(\alpha\beta\right)\left(xy\right)$.
\\
 (2). $\alpha\left(\beta x\right)=\alpha\left(\beta\left(\mathbf{1}x\right)\right)=\alpha\left(\left(\beta\mathbf{1}\right)x\right)=\left(\beta\mathbf{1}\right)\left(\alpha x\right)=\beta\left(\mathbf{1}\left(\alpha x\right)\right)=\beta\left(\alpha\left(\mathbf{1}x\right)\right)=\beta\left(\alpha x\right)$,
so $\left(\alpha\beta\right)x=\left(\beta\alpha\right)x$.$\qedhere$
\end{proof}
Note that the assumption that $R$ is commutative is not used in the
proofs given in Section 3, only the fact that $\alpha\beta$ and $\beta\alpha$
act in the same way on $X$, as stated in Lemma \ref{thm:lem1} (2).
The assumption that $\alpha\beta=\beta\alpha$ in $R$ was also not
used in the proof of Lemma \ref{thm:lem1} (2). 

\section{Ideals and additive group operations}

By removing the additive group of an $R$-algebra, we remove the three
group operations $\left(x,y\right)\mapsto x+y$, $\left(x\right)\mapsto-x$
and $\left(\right)\mapsto0$ (addition, taking the inverse and selecting
the group identity). It turns out, however, that under natural assumptions
the group structure can be partially recovered. Specifically, a scalable
monoid over $R$ can be partitioned into equivalence classes in each
of which a group structure corresponding to the additive group structure
in $R$ can often be defined.

\subsection{Commensurability and ideals}
\begin{defn}
\label{d3.1}Let $X$ be a scalable monoid over $R$, and let $\sim$
be a relation on $X$ such that $x\sim y$ if and only if $\alpha x=\beta y$
for some $\alpha,\beta\in R.$ Then $x$ and $y$ are said to be \emph{commensurable},
and the relation $\sim$ is called \emph{commensurability}. $\hphantom{\square}\square$ 
\end{defn}
As a trivial consequence of this definition, $x\sim\lambda x$ for
any $x\in X$ and $\lambda\in R$, since $\lambda x=1\left(\lambda x\right)$.
\newpage{}
\begin{prop}
\label{s3.1}Commensurability is an equivalence relation.
\end{prop}
\begin{proof}
The commensurability relation $\sim$ is reflexive because $1x=1x$,
it is symmetric because if $\alpha x=\beta y$ then $\beta y=\alpha x$,
and it is transitive because if $\alpha x=\beta y$ and $\gamma y=\delta z$
then $\gamma\left(\alpha x\right)=\gamma\left(\beta y\right)$ and
$\beta\left(\gamma y\right)=\beta\left(\delta z\right)$, so $\left(\gamma\alpha\right)x=\left(\beta\delta\right)z$
since $\gamma\left(\beta y\right)=\beta\left(\gamma y\right)$.$\qedhere$ 
\end{proof}
\begin{defn}
\label{d3.2}Let $X$ be a scalable monoid. An\emph{ ideal} in $X$
is an equivalence class with respect to commensurability. We denote
the ideal containing $x\in X$ by $\left[x\right]$. $\hphantom{\square}\square$
\end{defn}
\begin{prop}
\label{s3.2}Let $X$ be a scalable monoid. The commensurability relation
$\sim$ is a congruence relation on $X$, so the corresponding set
of equivalence classes can be made into an algebraic structure $X/{\sim}$
with a binary operation $\left(\left[x\right],\left[y\right]\right)\mapsto\left[x\right]\left[y\right]$
defined by 
\[
[x][y]=[xy]
\]
for any $x,y\in X$. $X/{\sim}$ is a monoid with $\left[\mathbf{1}\right]$
as unit element.
\end{prop}
\begin{proof}
If $\alpha x=\alpha'x'$ and $\beta y=\beta'y'$ then $(\alpha x)(\beta y)=(\alpha'x')(\beta'y')$,
so $(\alpha\beta)(xy)=(\alpha'\beta')(x'y')$. Thus, if $x\sim x'$
and $y\sim y'$ then $xy\sim x'y'$, so the product in $X/{\sim}$
given by $[x][y]=[xy]$ is well-defined. Also, for any $x\in X,$
$\mathbf{1}x=x\mathbf{1}=x$, so $\left[\mathbf{1}\right]\left[x\right]=\left[x\right]\left[\mathbf{1}\right]=\left[x\right]$,
so $\left[\mathbf{1}\right]$ is a unit element for $X/{\sim}$.$\qedhere$
\end{proof}
$X/{\sim}$ is the image of $X$ under the mapping $h:x\mapsto\left[x\right]$,
and it is clear that $h$ is a homomorphism of monoids. 

Note that a scaloid is akin to a tensor product of modules. Recall
that the tensor product of $u\in U$ and $v\in V$ is a tensor $u\otimes v$
belonging to $U\otimes V$; this is a module in general distinct from
$U$ and $V$. Similarly, the product $xy$ of $x$ and $y$ in a
scaloid $X$ does not in general belong to the same ideal as $x$
or $y$; in general $\left[x\right],\left[y\right]\neq\left[xy\right]$.
This suggests an application of scaloids: ``quantities'' such as
$1\,\mathsf{N}$, $2\,\mathsf{m}$ or $2\,\mathsf{Nm}$ can be regarded
as elements of a scaloid, specifically quantities with different ``sorts''
or ``dimensions'', corresponding to different equivalence classes
of such quantities, namely ideals.

Pursuing the analogy with tensor products, certain additional assumptions
allow us to regard each ideal as a module, making it possible to add
and subtract quantities as described below.

\subsection{Recovering additive group operations}
\begin{defn}
Let $X$ be a scalable monoid over $R$, and consider an ideal $\mathfrak{D}\subset X$.
A \emph{pivot} for $\mathfrak{D}$ is an element $p\in\mathfrak{D}$
such that for every $x\in\mathfrak{D}$ there is a unique $\lambda\in R$
such that $x=\lambda p$. $\hphantom{\square}\square$
\end{defn}
\begin{prop}
If $p$ and $p'$ are pivots for $\mathfrak{D}$, so that $x=\alpha p=\alpha'p'$
and $y=\beta p=\beta'p'$ for any $x,y\in\mathfrak{D}$, then $\left(\alpha+\beta\right)p=\left(\alpha'+\beta'\right)p'$. 
\end{prop}
\begin{proof}
As $p'\in\mathfrak{R}$, there is a unique $\lambda\in R$ such that
$p'=\lambda p$. Thus, $\left(\alpha'+\beta'\right)p'=\left(\alpha'+\beta'\right)\left(\lambda p\right)=\left(\left(\alpha'+\beta'\right)\lambda\right)p=\left(\alpha'\lambda+\beta'\lambda\right)p$.
Also, $x=\alpha p=\alpha'p'=\alpha'\left(\lambda p\right)$ implies
$\alpha=\alpha'\lambda$, and $y=\beta p=\beta'p'=\beta'\left(\lambda p\right)$
implies $\beta=\beta'\lambda$.$\qedhere$ 
\end{proof}
Hence, the sum of two elements of a scalable monoid can be defined
as follows: 
\begin{defn}
\label{d3.3}Let $X$ be a scalable monoid over $R$, and let $p$
be a pivot for an ideal $\mathfrak{D}$. If $x=\alpha p$ and $y=\beta p$,
we set 
\[
x+y=\left(\alpha+\beta\right)p.\hphantom{\square}\square
\]
\end{defn}
Note that if $x=\alpha p$ and $y=\beta p$ then $\beta x=\beta\left(\alpha p\right)=\alpha\left(\beta p\right)=\alpha y$,
so $x+y$ can be defined \emph{only if} $x\sim y$ \textendash{} one
cannot add incommensurable quantities.

As a trivial consequence of Definition \ref{d3.3}, addition in an
ideal is commutative.

If $p$ is a pivot for an ideal $\mathfrak{D}$ and $q\in\mathfrak{D}$
then $0q=0\left(\lambda p\right)=\left(0\lambda\right)p=0p$, so there
is a unique element $\mathbf{0}_{\mathfrak{D}}\in\mathfrak{D}$ such
that $\mathbf{0}_{\mathfrak{D}}=0q$ for any $q\in\mathfrak{D}$,
defined by 
\[
\mathbf{0}_{\mathfrak{D}}=0p
\]
for any pivot $p$ for $\mathfrak{D}$. It follows immediately from
Definition \ref{d3.3} that $x+\mathbf{0}_{\mathfrak{D}}=\mathbf{0}_{\mathfrak{D}}+x=x$
for any $x\in\mathfrak{D}$. 

A unital ring $R$ has a unique additive inverse $-1$ of $1_{R}$,
and we set 
\[
-x=\left(-1\right)x
\]
for all $x\in X$. If $\mathfrak{D}$ has a pivot, we also set $x-y=x+\left(-y\right)$
for any $x,y\in\mathfrak{D}$, so that $x-x=-x+x=\mathbf{0}_{\mathfrak{D}}$
for any $x\in\mathfrak{D}$.

\section{Free commutative scalable monoids}

A commutative scalable monoid $X$ is one where $xy=yx$ for all $x,y\in X$.
In this section, only commutative scalable monoids will be considered. 
\begin{defn}
\label{d2.3}Let $X$ be a commutative scalable monoid over $R$.
A (finite) \emph{basis} for $X$ is a set $B=\left\{ b_{1},\ldots,b_{n}\right\} $
of invertible elements of $X$ such that every $x\in X$ has a unique
(up to order of factors) expansion 
\[
x=\mu\prod_{i=1}^{n}b_{i}^{_{_{k_{i}}}},
\]
where $\mu\in R$ and $k_{1},\dots,k_{n}$ are integers. $\hphantom{\square}\square$ 
\end{defn}
Any product of invertible quantities is invertible, so any product
of basis elements is invertible. As $X$ is commutative, we have
\[
\left(\mu\prod_{i=1}^{n}b_{i}^{_{_{k_{i}}}}\right)\left(\nu\prod_{i=1}^{n}b_{i}^{\ell_{i}}\right)=\left(\mu\nu\right)\prod_{i=1}^{n}b_{i}^{_{\left(k_{i}+\ell_{i}\right)}}.
\]

\subsection{Quantity spaces}

Below, only scaloids over a field $K$ will be considered. 
\begin{defn}
\label{d2.4}A (finite-rank) \emph{free commutative scalable monoid
over a field} $K$, or \emph{a quantity space} over $K$, is a commutative
scalable monoid over $K$ which has a (finite) basis. $\hphantom{\square}\square$

With a view to applications, elements of a quantity space may be called
\emph{quantities,} and ideals in a quantity space may be called \emph{dimensions}.
As explained in \cite{key-4}, a basis for a quantity space can be
interpreted as a \emph{system of fundamental units of measurement}.
\end{defn}
\begin{lem}
\label{s3.3}Let $X$ be a quantity space over $K$ with a basis $\left\{ b_{1},\ldots b_{n}\right\} $,
and consider $x=\mu\prod_{i=1}^{n}b_{i}^{_{_{k_{i}}}}$ and $y=\nu\prod_{i=1}^{n}b_{i}^{_{_{\ell_{i}}}}$.
The following conditions are equivalent: 
\begin{enumerate}
\item $x\sim y.$ 
\item $k^{i}=\ell{}^{i}$ for $i=1,\ldots,n$. 
\item $\nu x=\mu y$, or equivalently $\mu_{B}\left(y\right)x=\mu_{B}\left(x\right)y$. 
\end{enumerate}
\end{lem}
\begin{proof}
$(1)\Rightarrow(2)$. If $x\sim y$ then $\left(\alpha\mu\right)\prod_{i=1}^{n}b_{i}^{_{_{k_{i}}}}=z=\left(\beta\nu\right)\prod_{i=1}^{n}b_{i}^{_{_{\ell_{i}}}}$
for some \linebreak{}
 $\alpha,\beta\in K$. As the expansion of $z$ is unique, $k^{i}=\ell{}^{i}$
for $i=1,\ldots,n$.

The implications $(2)\Rightarrow(3)$ and $(3)\Rightarrow(1)$ are
trivial.$\qedhere$ 
\end{proof}

\subsection{The measure of a quantity}
\begin{defn}
\label{d2.5}Let $X$ be a quantity space over $K$, and let $B=\left\{ b_{1},\ldots,b_{n}\right\} $
be a basis for $X$. The uniquely determined scalar $\mu\in K$ in
the expansion 
\[
x=\mu\prod_{i=1}^{n}b_{i}^{k_{i}}
\]
is called the \emph{measure} of $x$ relative to \textbf{$B$} and
will be denoted by $\mu_{B}\left(x\right)$. $\hphantom{\square}\square$ 
\end{defn}
For example, $\mu_{B}\left(\mathbf{1}\right)=1$ for any $B$, because
$\mathbf{1}=1\cdot\mathbf{1}=1\prod_{i=1}^{n}b_{i}^{0}$ for any $B$.

Relative to a fixed basis, measures of quantities can be used as proxies
for the quantities themselves. For example, the measure of a product
of quantities is equal to the product of the measures of these quantities. 
\begin{prop}
\label{s2.1}Let $B=\left\{ b_{1},\ldots,b_{n}\right\} $ be a basis
for a quantity space $X$ over $K$. 
\begin{enumerate}
\item For any $x\in X$ and $\lambda\in K$, $\mu_{B}\left(\lambda x\right)=\lambda\mu_{B}\left(x\right)$. 
\item For any $x,y\in X$, $\mu_{B}\left(xy\right)=\mu_{B}\left(x\right)\mu_{B}\left(y\right)$. 
\end{enumerate}
\end{prop}
\begin{proof}
We have $x=\mu\prod_{i=1}^{n}b{}_{i}^{k_{i}}$ and $y=\nu\prod_{i=1}^{n}b_{i}^{\ell_{i}}$,
where $\mu,\nu\in K$. \\
 (1). $\lambda x=\lambda\left(\mu\prod_{i=1}^{n}b{}_{i}^{k_{i}}\right)=\left(\lambda\mu\right)\prod_{i=1}^{n}b{}_{i}^{k_{i}}$,
so $\mu_{B}\left(\lambda x\right)=\lambda\mu=\lambda\mu_{B}\left(x\right)$.
\\
 (2). $\left(\mu\prod_{i=1}^{n}b_{i}^{_{_{k_{i}}}}\right)\left(\nu\prod_{i=1}^{n}b_{i}^{\ell_{i}}\right)=\left(\mu\nu\right)\prod_{i=1}^{n}b_{i}^{_{\left(k_{i}+\ell_{i}\right)}}$,
so $\mu_{B}\left(xy\right)=\mu\nu=$$\mu_{B}\left(x\right)\mu_{B}\left(y\right)$.
$\qedhere$ 
\end{proof}
\begin{prop}
\label{s2.6}A quantity $x\in X$ is invertible if and only if $\mu_{B}\left(x\right)\neq0$,
and for any invertible $x\in X$ we have $\mu_{B}\left(x^{-1}\right)=1/\mu_{B}\left(x\right)$.
\end{prop}
\begin{proof}
If $\mu_{B}\left(x\right)\neq0$ and $x=\mu_{B}\left(x\right)\prod_{i=1}^{n}b_{i}^{k_{i}}$
then 
\[
\frac{\prod_{i=1}^{n}b_{i}^{-k_{i}}}{\mu_{B}\left(x\right)}x=x\frac{\prod_{i=1}^{n}b_{i}^{-k_{i}}}{\mu_{B}\left(x\right)}=\left(\mu_{B}\left(x\right)\prod_{i=1}^{n}b_{i}^{k_{i}}\right)\frac{\prod_{i=1}^{n}b_{i}^{-k_{i}}}{\mu_{B}\left(x\right)}=\mathbf{1},
\]
so $x$ is invertible. If, conversely, $x$ has an inverse $x^{-1}$
then $\mu_{B}\left(x\right)\mu_{B}\left(x^{-1}\right)=\mu_{B}\left(xx^{-1}\right)=\mu_{B}\left(\mathbf{1}\right)=1$,
so $\mu_{B}\left(x\right)\neq0$ and $\mu_{B}\left(x^{-1}\right)=1/\mu_{B}\left(x\right)$.$\qedhere$
\end{proof}
\begin{prop}
\label{s3.4}Let $X$ be a quantity space over $K.$ For every $x\in\left[\mathbf{1}\right]$,
$\mu_{B}\left(x\right)$ does not depend on $B$.
\end{prop}
\begin{proof}
Let $B=\left\{ b_{1},\ldots,b_{n}\right\} $ and $\widehat{B}=\left\{ \widehat{b}_{1},\ldots,\widehat{b}_{m}\right\} $
be bases for $X$. The quantity $\mathbf{1}$ has the expansions $\mathbf{1}=1\prod_{i=1}^{n}b_{i}^{0}$
and $\mathbf{1}=1\prod_{i=1}^{m}\widehat{b}_{i}^{0}$. In view of
Lemma \ref{s3.3}, therefore, $x=\mu_{B}\left(x\right)\prod_{i=1}^{n}b_{i}^{0}$
and $x=\mu_{\widehat{B}}\left(x\right)\prod_{i=1}^{m}\widehat{b}_{i}^{0}=\mu_{\widehat{B}}\left(x\right)\left(1\prod_{i=1}^{m}\widehat{b}_{i}^{0}\right)=\mu_{\widehat{B}}\left(x\right)\cdot\mathbf{1}=\mu_{\widehat{B}}\left(x\right)\left(1\prod_{i=1}^{n}b_{i}^{0}\right)=\mu_{\widehat{B}}\left(x\right)\prod_{i=1}^{n}b_{i}^{0}$,
so $\mu_{B}\left(x\right)=\mu_{\widehat{B}}\left(x\right)$.$\qedhere$ 
\end{proof}
It is common to refer to a quantity $x\in\left[\mathbf{1}\right]$
as a ``dimensionless quantity'', although $x$ is not really ``dimensionless''
\textendash{} it belongs to, or ``has'', the dimension $\left[\mathbf{1}\right]$.
The fact that the measure of any $x\in\left[\mathbf{1}\right]$ does
not depend on a choice of basis \textendash{} that is, a choice of
fundamental units of measurement \textendash{} is the foundation for
dimensional analysis \cite{key-2,key-1,key-4}.

\subsection{Additive group operations in quantity spaces}
\begin{prop}
\label{s3.5-1}Let $X$ be a quantity space over $K$. 
\begin{enumerate}
\item If $x,y\in X$, $x\sim y$ and $x$ is invertible then $y=\kappa x$
for some $\kappa\in K$. 
\item \label{s2.2-1}If $x\in X$ is invertible and $\lambda x=\lambda'x$
for some $\lambda,\lambda'\in K$ then $\lambda=\lambda'$. 
\end{enumerate}
\end{prop}
\begin{proof}
(1). $\mu_{B}\left(y\right)x=\mu_{B}\left(x\right)y$ by Lemma \ref{s3.3},
and $\mu_{B}\left(x\right)\neq0$ by Proposition \ref{s2.6}, so 
\[
y=\frac{\mu_{B}\left(y\right)}{\mu_{B}\left(x\right)}x.
\]
(2). Let $x$ have the expansion $x=\mu\prod_{i=1}^{n}b{}_{i}^{k_{i}}$
relative to a basis $\left\{ b_{1},\ldots,b_{n}\right\} $ for $X$.
If $\lambda x=\lambda'x$ then $\lambda\mu\prod_{i=1}^{n}b_{i}^{k_{i}}=z=\lambda'\mu\prod_{i=1}^{n}b_{i}^{k_{i}}$,
so $\lambda\mu=\lambda'\mu$ since the expansion of $z$ is unique,
and $\mu\neq0$ since $x$ is invertible, so $\lambda=\lambda'$.
$\qedhere$ 
\end{proof}
Thus, every invertible quantity is a pivot for the dimension to which
it belongs. This has an important consequence: 
\begin{cor}
Every dimension in a quantity space has a pivot.
\end{cor}
\begin{proof}
An equivalence class is not empty, so for every dimension $\mathfrak{D}$
there is some quantity $\mu\prod_{i=1}^{n}b{}_{i}^{k_{i}}\in\mathfrak{D}$,
so there is an invertible element $1\prod_{i=1}^{n}b{}_{i}^{k_{i}}\in\mathfrak{D}$.$\qedhere$ 
\end{proof}
In a quantity space, $x+y$ is thus defined \emph{if and only if}
$x\sim y$. Addition of quantities in a dimension relates to addition
of measures in the following way: 
\begin{prop}
\label{s3.5}Let $X$ be a quantity space over $K$. For any basis
$B$ for $X$ and $x,y\in X$ such that $x\sim y$, we have $\mu_{B}\left(x\right)+\mu_{B}\left(y\right)=\mu_{B}\left(x+y\right)$.
\end{prop}
\begin{proof}
Let $x=\mu_{B}\left(x\right)\prod_{i=1}^{n}b_{i}^{k_{i}}$ and $y=\mu_{B}\left(y\right)\prod_{i=1}^{n}b_{i}^{k_{i}}$
be the expansions of $x$ and $y$ relative to $B=\left\{ b_{1},\ldots.b_{n}\right\} $.
$\prod_{i=1}^{n}b_{i}^{k_{i}}$ is invertible, and thus a pivot, so
\[
x+y=\mu_{B}\left(x\right)\prod_{i=1}^{n}b_{i}^{k_{i}}+\mu_{B}\left(y\right)\prod_{i=1}^{n}b_{i}^{k_{i}}=\left(\mu_{B}\left(x\right)+\mu_{B}\left(y\right)\right)\prod_{i=1}^{n}b_{i}^{k_{i}},
\]
so we have obtained the unique expansion of $x+y$ relative to $B$,
and this expansion shows that $\mu_{B}\left(x+y\right)=\mu_{B}\left(x\right)+\mu_{B}\left(y\right)$.$\qedhere$ 
\end{proof}
In words, the measure of a sum of quantities is equal to the sum of
the measures of these quantities. Also, $\mu_{B}\left(\lambda x\right)=\lambda\mu_{B}\left(x\right)$,
so measures represent quantities in a given dimension in the same
way that coordinates represent vectors.

\subsection{Groups of dimensions; cardinality of bases}

If $X$ is commutative then $X/{\sim}$ is also commutative, and in
this section we prove that $X/{\sim}$ is actually a free abelian
group, not only a commutative monoid. 
\begin{prop}
Let $X$ be a quantity space over $K$. For every dimension $\left[x\right]\in X/{\sim}$
there is a unique dimension $\left[x\right]^{-1}\in X/{\sim}$ such
that $\left[x\right]\left[x\right]^{-1}=\left[x\right]^{-1}\left[x\right]=\left[\mathbf{1}\right]$.
\end{prop}
\begin{proof}
Let $x=\mu\prod_{i=1}^{n}b_{i}^{k_{i}}$ be the unique expansion of
$x$ relative to the basis $\left\{ b_{1},\ldots,b_{n}\right\} $,
and set $y=1\prod_{i=1}^{n}b_{i}^{k_{i}}$ and $z=1\prod_{i=1}^{n}b_{i}^{-k_{i}}$.
(By Lemma \ref{s3.3}, $y$ depends on $\left[x\right]$, but not
on its representative $x$.) Then $\left[x\right]=\left[y\right]$
and $\left[y\right]\left[z\right]=\left[z\right]\left[y\right]=\left[zy\right]=\left[\mathbf{1}\right]$,
so $\left[x\right]\left[z\right]=\left[z\right]\left[x\right]=\left[\mathbf{1}\right]$.
Furthermore, if $\left[x\right]^{-1}\left[x\right]=\mathbf{1}$ then
$\left[x\right]^{-1}=\left[x\right]^{-1}\left(\left[x\right]\left[z\right]\right)=\left(\left[x\right]^{-1}\left[x\right]\right)\left[z\right]=\left[z\right]$.$\qedhere$

\newpage{}
\end{proof}
Recall that a (finite) basis for an abelian group $G$ is a set $\left\{ b_{1},\ldots,b_{n}\right\} $
of elements of $G$ such that every $x\in G$ has a unique expansion
$x=\prod_{i=1}^{n}b_{i}^{k_{i}}$, and that a free (finitely generated)
abelian group is an abelian group for which such a basis exists.
\begin{prop}
\label{s3.7}Let $X$ be a quantity space over $K$. 
\begin{enumerate}
\item If $B=\left\{ b_{1},\ldots,b_{n}\right\} $ is a basis for $X$, then
$B^{*}=\left\{ \left[b_{1}\right],\ldots,\left[b_{n}\right]\right\} $
is a basis with the same cardinality for $X/{\sim}$. 
\item Conversely, if $B^{*}=\left\{ \left[b_{1}\right],\ldots,\left[b_{n}\right]\right\} $,
where each $b_{i}$ is invertible, is a basis for $X/{\sim}$, then
$B=\left\{ b_{1},\ldots,b_{n}\right\} $ is a basis with the same
cardinality for $X$. 
\end{enumerate}
\end{prop}
\begin{proof}
The unique expansions of $b_{i},b_{i'}\in B$ relative to $B$ are
$b_{i}=1b_{i}$ and $b_{i'}=1b_{i'}$, respectively, so $\mu_{B}\left(b_{i}\right)=\mu_{B}\left(b_{i'}\right)=1$.
Hence, $\left[b_{i}\right]=\left[b_{i'}\right]$ implies $b_{i}=b_{i'}$
since $b_{i}\sim b_{i'}$ implies $\mu_{B}\left(b_{i'}\right)b_{i}=\mu_{B}\left(b_{i}\right)b_{i'}$
by Lemma \ref{s3.3}, so the surjective mapping $b_{i}\mapsto\left[b_{i}\right]$
is injective as well.\\
(1). Let $\left[x\right]$ be an arbitrary dimension in $X/{\sim}$.
As $B$ generates $X$, $x=\mu\prod_{i=1}^{n}b_{i}^{k_{i}}$ for some
$\mu\in K$ and some integers $k_{1},\ldots,k_{n}$, so $\left[x\right]=\left[\mu\prod_{i=1}^{n}b_{i}^{k_{i}}\right]=\left[\prod_{i=1}^{n}b_{i}^{k_{i}}\right]=\prod_{i=1}^{n}\left[b_{i}\right]^{k_{i}}$,
so $B^{*}$ generates $X/{\sim}$.

Also, if $\left[x\right]=\prod_{i=1}^{n}\left[b_{i}\right]^{k_{i}}=\prod_{i=1}^{n}\left[b_{i}\right]^{\ell{}_{i}}$,
then $\left[\prod_{i=1}^{n}b_{i}^{k_{i}}\right]=\left[\prod_{i=1}^{n}b_{i}^{\ell{}_{i}}\right]$,
so $k_{i}=\ell_{i}$ for $i=1,\ldots,n$ by Lemma \ref{s3.3}.\\
 (2). Let $x$ be an arbitrary quantity in $X$. As $B^{*}$ generates
$X/{\sim}$, we have $\left[x\right]=\prod_{i=1}^{n}\left[b_{i}\right]^{k_{i}}=\left[\prod_{i=1}^{n}b_{i}^{k_{i}}\right]$,
and as $\prod_{i=1}^{n}b_{i}^{k_{i}}$ is invertible, Proposition
\ref{s3.5-1} (1) implies that there exists some $\kappa\in K$ and
integers $k_{1},\ldots,k_{n}$ such that $x=\kappa\prod_{i=1}^{n}b_{i}^{k_{i}}$.

Finally, if $x=\mu\prod_{i=1}^{n}b_{i}^{k_{i}}=\nu\prod_{i=1}^{n}b_{i}^{\ell_{i}}$
then $\left[\mu\prod_{i=1}^{n}b_{i}^{k_{i}}\right]=\left[\nu\prod_{i=1}^{n}b_{i}^{\ell_{i}}\right]$,
so $\left[\prod_{i=1}^{n}b_{i}^{k_{i}}\right]=\left[\prod_{i=1}^{n}b_{i}^{\ell_{i}}\right]$,
so $\prod_{i=1}^{n}\left[b_{i}\right]^{k_{i}}=\prod_{i=1}^{n}\left[b_{i}\right]^{\ell_{i}}$,
so $k_{i}=\ell_{i}$ for $i=1,\ldots,n$, since $B^{*}$ is a basis
for $X/\sim$. Also, $x=\mu y=\nu y$, where $y$ is invertible, so
$\mu=\nu$ by Proposition \ref{s3.5-1} (2).$\qedhere$
\end{proof}
\begin{cor}
\label{s3.8}Let $X$ be a quantity space over $K$. 
\begin{enumerate}
\item $X/{\sim}$ is a free abelian group. 
\item Any two bases for $X/{\sim}$ have the same number of elements, any
basis for $X$ has the same number of elements as any basis for $X/{\sim}$,
and any two bases for $X$ have the same number of elements. 
\end{enumerate}
\end{cor}
\begin{proof}
(1) is immediate. To prove (2), it suffices to note that any two bases
for a free abelian group have the same cardinality.$\qedhere$ 
\end{proof}

\subsection{Concerning algebraic structures acting on monoids of quantities}

In theoretical discussions about measurement, the measure of a physical
quantity is usually considered to be a real number, so it is natural
to let the field associated with a quantity space be the real numbers
$\mathbb{R}$. However, some physical quantities, such as a distance
or a mass, can only have positive or non-negative measures. This suggests
that the field acting on the monoid of quantities should be replaced
by a more general structure. In \cite{key-4}, the concept of a scaloid
over a field is replaced by the concept of a scaloid over a so-called
\emph{scalar system.} A scalar system can be conveniently defined
as a subset of a field, inheriting addition and multiplication in
the field, such that it is closed under addition and its non-zero
elements constitute a group under multiplication. The real numbers,
the non-negative real numbers, and the positive real numbers are obvious
examples of scalar systems.

There is, however, a way to accommodate constrained sets of measures
without modifying the notion of a quantity space: we let functions
of the form $\mu_{B}:\mathfrak{D}\rightarrow\mathbb{R}$ be \emph{partial
functions}, defined on some subset $\mathfrak{S}$ of $\mathfrak{D}$,
so that, for example, $\mu_{B}\left(x\right)>0$ for all $x\in\mathfrak{S}$.
This approach is quite flexible; it also works, for example, when
we require $\mu_{B}\left(\mathfrak{S}\right)$ to be a discrete set
of numbers.

It is an elementary observation that physical quantities are ordered;
three grams is more than two grams, and so on. This means that it
should be possible to equip a quantity space with some kind of order
relation, not considered so far in this article. Such an order relation
can be induced by the order on an ordered field acting on the monoid
of quantities.

Specifically, let $X$ be a quantity space over an ordered field $K$.
To induce an order relation, based on the order on $K$, on $X$,
we need to to choose an orientation for $X$. Analogous to how we
proceed for vector spaces, we choose an orientation by fixing a basis
$B=\left\{ b_{1},\ldots,b_{n}\right\} $ for $X$ and stipulating
that for any $\mathfrak{D}\in X/{\sim}$ and $x\in\mathfrak{D}$ we
have $x\geq0_{\mathfrak{D}}$ if and only if there are integers $k_{1},\ldots,k_{n}$
and $\mu\in K$ such that $x=\mu\prod_{i=1}^{n}b_{i}^{k_{i}}$ where
$\mu\geq0$. Moreover, we stipulate that $x\geq y$ if and only if
$\left(x-y\right)\geq0_{\mathfrak{D}}$.

In view of this construction, the fact that $\mathbb{R}$ is an ordered
field gives another reason for considering quantity spaces over $\mathbb{R}$.

\section{In conclusion}

As shown in the Introduction, scalable monoids complement rings and
modules from an abstract mathematical point of view. From the point
of view of applications, a quantity space is a natural counterpart
to a vector space: quantities and vectors are both fundamental notions
in physics and quantitative sciences generally. Recall that the transformation
of vector space theory into an axiomatic, ``coordinate-free'' form
was completed during the interwar period in Europe, more than three
quarters of a century ago, but the corresponding formulation of an\linebreak{}
 axiomatic, ``coordinate-free'' notion of quantities has not yet
been completed, in my opinion, despite important contributions (e.g.,
\cite{key-3,key-6,key-7}).\linebreak{}
 Indeed, the research reported here started as an attempt to model
quantities, not as an attempt to fill a mathematical lacuna.

A crucial feature of systems of quantities is that two quantities
can be added if and only if they are ``commensurable'', so mathematical
models of systems of quantities should reflect this peculiar property,
not shared by numbers without sorts. An algebraic structure satisfying
this requirement can of course be found if sufficiently complicated
algebraic structures are considered. However, it would seem to be
desirable that systems of quantities be modeled by a formal mathematical
structure defined in a simple, direct manner, similar to the definition
of a vector space, and in such a way that the crucial ``commensurability''
feature is a natural consequence of this formulation. The definitions
given here would seem to pass this test.

Applications of quantity spaces, considered only informally and superficially
here, are discussed at some length in \cite{key-4,key-5}. 

\newpage{}

\end{document}